\documentclass[11pt,a4paper]{amsart} 
\usepackage{amssymb,amscd,amsmath, young,mathrsfs}
\usepackage{mathrsfs, mathdots} 
\usepackage{float}
\usepackage[usenames]{color}
\usepackage{soul}
\usepackage{diagrams}
\usepackage{array,longtable}
\newcolumntype{C}{>{$}c<{$}}
\newcolumntype{R}{>{$}r<{$}}
\newcolumntype{L}{>{$}l<{$}}

\theoremstyle{plain}
\newtheorem{thm}{Theorem}[section]
\newtheorem{lem}[thm]{Lemma}
\newtheorem{pro}[thm]{Proposition}

\newtheorem{lemma}[thm]{Lemma}

\newtheorem{asm}[thm]{Assumption}

\theoremstyle{remark}
\newtheorem{rem}[thm]{Remark}

\newtheorem{exm}[thm]{Example}

\newtheorem{dfn}[thm]{Definition}
\newtheorem*{question*}{Question}

\newtheorem*{acknowledgements}{Acknowledgements}

\newcommand{\N}{\mathbb{N}}
\newcommand{\Z}{\mathbb{Z}}
\newcommand{\Q}{\mathbb{Q}}
\newcommand{\F}{\mathbb{F}}

\newcommand{\tensor}{\otimes}

\renewcommand{\epsilon}{\varepsilon}

\renewcommand{\phi}{\varphi}

\DeclareMathOperator{\GL}{GL}

\DeclareMathOperator{\ind}{ind}
\DeclareMathOperator{\Ind}{Ind}

\author{Michael M.~Schein} \address{Department of Mathematics,
  Bar-Ilan University, Ramat Gan 5290002,
  Israel}\email{mschein@math.biu.ac.il}
\date{\today}

\begin{document}
 \title[Supersingular representations of $\mathrm{GL}_2(F)$]{A family of irreducible supersingular representations of $\mathrm{GL}_2(F)$ for some ramified $p$-adic fields} 

\begin{abstract}
We construct infinite families of irreducible supersingular mod $p$ representations of $\mathrm{GL}_2(F)$ with $\mathrm{GL}_2(\mathcal{O}_F)$-socle compatible with Serre's modularity conjecture, where $F / \Q_p$ is any finite extension with residue field $\mathbb{F}_{p^2}$ and ramification degree $e \leq (p-1)/2$.  These are the first such examples for ramified $F / \Q_p$.  
\end{abstract}

\maketitle

\section{Introduction}
\subsection{Supersingular representations}
The irreducible smooth mod $p$ representations of $\mathrm{GL}_2(F)$ admitting a central character, where $F / \Q_p$ is a finite extension, were classified by Barthel and Livn\'{e}~\cite[Theorems~33-34]{BL/94}, except for the {\emph{supersingular}} representations.  The irreducible supersingular representations of $\mathrm{GL}_2(\Q_p)$ were determined by Breuil~\cite[Th\'{e}or\`{e}me~1.1]{Breuil/03}.  By contrast, little is known about the irreducible supersingular representations of $\mathrm{GL}_2(F)$ when $F \neq \Q_p$.  Wu~\cite[Theorem~1.1]{Wu/21} has shown that they do not have finite presentation; this was proved earlier by Schraen~\cite[Th\'{e}or\`{e}me~2.23]{Schraen/15} in the case $[F : \Q_p] = 2$.  Thus a direct construction of such representations appears to be difficult.  However, the method of diagrams introduced by Pa\v{s}k\={u}nas~\cite{Paskunas/04} can be used to show the existence of supersingular representations with certain properties; indeed, such representations of $\mathrm{GL}_2(F)$, for arbitrary $F$, were produced in~\cite[Theorem~6.25]{Paskunas/04}.

A primary motivation for studying  
supersingular representations of $\mathrm{GL}_2(F)$ is their appearance in a hypothetical mod $p$ local Langlands correspondence.  The relevant representations for this purpose 
have a $\mathrm{GL}_2(\mathcal{O}_F)$-socle that is compatible with Serre's modularity conjecture; see Section~\ref{sec:good.socle} below.  We refer to such representations as having {\emph{good socle}}.  While the supersingular representations constructed in~\cite{Paskunas/04} only have good socle when $F = \Q_p$, Breuil and Pa\v{s}k\={u}nas constructed infinite families of diagrams giving rise to supersingular irreducible representations of $\mathrm{GL}_2(F)$ with good socle for all unramified $F/\Q_p$ with $p > 2$.  These families are parametrized by choices of a collection of isomorphisms between one-dimensional $\overline{\mathbb{F}}_p$-vector spaces.  In particular, there are far more irreducible mod $p$ representations of $\mathrm{GL}_2(F)$ than there are mod $p$ representations of the absolute Galois group of $F$.  While the families considered in~\cite{BP/12} are not exhaustive, they include supersingular representations arising from global constructions that are believed to realize the mod $p$ local Langlands correspondence; see~\cite{Breuil/14, BHHMS/21, DL/21} and references therein.

\subsection{Good socles} \label{sec:good.socle}
Let $\mathcal{O}$ be the ring of integers of $F$, and let $k$ be its residue field.  Set $G = \mathrm{GL}_2(F)$ and $K = \mathrm{GL}_2(\mathcal{O})$.  
A \emph{Serre weight} is an irreducible $\overline{\F}_p$-representation of the finite group $\Gamma = \mathrm{GL}_2(k)$, which can be viewed as a representation of $K$ by inflation.  Every irreducible smooth $\overline{\F}_p$-representation of $K$ arises from a Serre weight in this way.  Let $\Sigma$ be the set of Serre weights.  Write $G_F = \mathrm{Gal}(\overline{\Q}_p/F)$, and let $\rho$ be a continuous two-dimensional representation of $G_F$ defined over a sufficiently large extension of $\mathbb{F}_p$.  The mod $p$ local Langlands correspondence is expected to associate to $\rho$ a mod $p$ representation $\pi(\rho)$ of $G$ whose socle as a $K$-module is
$$ \mathrm{soc}_K \pi(\rho) = \bigoplus_{\sigma \in \Sigma} \sigma^{\oplus \mu_{\sigma}(\rho)},$$
where the $\mu_{\sigma}(\rho)$ are the ``intrinsic multiplicities'' whose existence is the content of the generalized Breuil-M\'{e}zard conjecture~\cite[Conjecture~2.1.5]{GK/14}; see also~\cite[Conjecture~2.2.3]{Kisin/10}.  In this paper we consider irreducible $\rho$; in this case $\pi(\rho)$ is expected to be supersingular and irreducible.  It is also expected~\cite[Conjecture~3.2.7]{GHS/18} that $\mu_{\sigma}(\rho) > 0$ if and only if $\sigma \in \mathcal{D}(\rho)$, where $\mathcal{D}(\rho)$ is the set of Serre weights associated to $\rho$ by the weight part of Serre's modularity conjecture.  For semisimple $\rho$, the set $\mathcal{D}(\rho)$ is given explicitly in~\cite[\S2]{Schein/08}; see~\cite[Lemma~4.22]{BLGG/13} and~\cite[Example~7.1.7]{GHS/18} for more modern interpretations of $\mathcal{D}(\rho)$.  In the situations considered in this paper, it is either known~\cite{GK/14, CDM/18} or believed that $\mu_{\sigma}(\rho) = 1$ for all $\sigma \in \mathcal{D}(\rho)$; we thus seek mod $p$ representations $V$ of $G$ whose $K$-socle has the form
\begin{equation} \label{equ:good.socle}
 \mathrm{soc}_K V = \bigoplus_{\sigma \in \mathcal{D}(\rho)} \sigma.
\end{equation}

\subsection{Results}
The main result of this paper is the following.  See Definition~\ref{def:generic} below for the genericity condition imposed on $\rho$, noting that it implies that $p > 2$ and that $F / \Q_p$ has ramification degree $e \leq (p-1)/2$.
\begin{thm} \label{thm:intro}
Let $F / \Q_p$ be a finite extension with residue field $\mathbb{F}_{p^2}$.  If $\rho$ is generic, then there exists an explicit infinite family of diagrams giving rise to irreducible supersingular representations $V$ of $\mathrm{GL}_2(F)$ satisfying~\eqref{equ:good.socle}. 
\end{thm}

To the author's knowledge this is the first construction of supersingular representations with good socle for ramified $p$-adic fields.  If $F / \Q_p$ is the quadratic unramified extension, then the supersingular representations associated to our diagrams coincide with those obtained by Breuil and Pa\v{s}k\={u}nas.  The hypotheses of Theorem~\ref{thm:intro} are discussed at the end of Section~\ref{subsec:overview} and in Remark~\ref{rem:reasons} below.  

A curious and rather unexpected feature of our diagrams is that, besides depending on choices of isomorphisms analogous to those in~\cite{BP/12}, they depend on the choice of an element in a large finite set, namely the set of Hamiltonian walks on an $e \times e$ square lattice.  The number of such Hamiltonian walks grows exponentially in $e^2$ and has been studied mostly by physicists; see~\cite[\S9]{BMGJ/05} and numerical computations, as well as predictions based on conformal field-theoretic models, in~\cite[\S XD]{JK/98} and~\cite{Jacobsen/07}.

As a consequence of the ``breakage of symmetry'' involved in the choice of a Hamiltonian walk, the representations of Theorem~\ref{thm:intro} appear unlikely to contribute to the mod $p$ local Langlands correspondence when $F / \Q_p$ is ramified; see Section~\ref{sec:Langlands} below.  The present work shows that the collection of irreducible supersingular representations of $\mathrm{GL}_2(F)$ with good socle has, in general, an even more complex structure than was discovered in~\cite{BP/12}.

\subsection{Overview of the construction} \label{subsec:overview}
In order to motivate our construction, we briefly review some aspects of the earlier work of Breuil and Pa\v{s}k\={u}nas.  Fix a uniformizer $\pi \in \mathcal{O}$.  Let $Z \leq G$ be the center, let $I \leq K$ be the Iwahori subgroup of residually upper triangular matrices, and let $N = N_G(I)$ be its normalizer in $G$.  Recall~\cite[Definition~5.14]{Paskunas/04} that a diagram for $F$ is a triple $(D_0, D_1, \iota)$, where $D_0$ is a smooth mod $p$ representation of $KZ$, whereas
$D_1$ is a smooth mod $p$ representation of $N$ and
$\iota: D_1 \to D_0$ is an $IZ$-equivariant homomorphism.
Suppose that $D_0$ is admissible, that the element $\pi \mathrm{Id}_2 \in Z$ acts trivially, and that $\iota$ is injective.  Then (in the case $p = 2$, which is irrelevant for this paper, a further technical hypothesis~\cite[Definition~9.1]{BP/12} is needed) the injective envelope ${\Omega}$ of $\mathrm{soc}_K D_0$ in the category of $K$-modules may be endowed with an action of $G$ such that $D_0 \subset {\Omega}_{| KZ}$ and $D_1 \subset {\Omega}_{|N}$~\cite[Theorem~9.8]{BP/12}.  

Any irreducible $\overline{\mathbb{F}}_p$-representation of $K$ necessarily factors through the finite quotient $\mathrm{GL}_2(k)$ and is thus a Serre weight; see \S\ref{sec:prelim} below.  
When $F / \Q_p$ is unramified and $\rho$ is generic, Breuil and Pa\v{s}k\={u}nas defined a family of diagrams for which $D_0$ is the maximal representation of $\Gamma = \mathrm{GL}_2(k)$ such that $\mathrm{soc}_{\Gamma} D_0 = \bigoplus_{\sigma \in \mathcal{D}(\rho)} \sigma$ and that no element of $\mathcal{D}(\rho)$ appears as a subquotient of $D_0 / \mathrm{soc}_{\Gamma} D_0$.  
It turns out that $D_0$ decomposes as $D_0 = \bigoplus_{\sigma \in \mathcal{D}(\rho)} D_{0,\sigma}(\rho)$, where $\mathrm{soc}_{\Gamma} D_{0, \sigma}(\rho) = \sigma$.  
We view $D_0$ as a $KZ$-module by inflation, with $\pi \mathrm{Id}_2$ acting trivially.  The $G$-submodule $\Psi \subset {\Omega}$ generated by $D_0$ is irreducible and supersingular.  A key step in the proof of this is the claim that if $W \subseteq \Psi$ is a $G$-submodule such that $W \cap D_{0,\sigma}(\rho) \neq 0$ for some $\sigma \in \mathcal{D}(\rho)$, then $D_{0,\sigma}(\rho) \subset W$; this follows from a rather technical analysis of the structure of $D_{0,\sigma}(\rho)$ and of maps of the form $\mathrm{ind}_{KZ}^G \tau \to W$, for $\tau \in \mathcal{D}(\rho)$, arising from Frobenius reciprocity.  The analogous claim fails when $F/\Q_p$ is ramified, for an analogous definition of $D_0$, since $\mathrm{ind}_{KZ}^G \tau$ does not contain enough non-split extensions of Serre weights factoring through $\Gamma$.

We now briefly describe the present work; the details are found in the body of the paper.
Let $\Q_p \subseteq F_0 \subseteq F$ be the maximal unramified subextension, and let $f = [F_0 : \Q_p]$.  Given an irreducible generic $\rho$ as above, there is a set $\mathcal{S}$ of irreducible generic Galois representations ${\rho}^\prime : G_{F_0} \to \mathrm{GL}_2(\overline{\mathbb{F}}_p)$ such that $| \mathcal{S} | = e^f$ and $\mathcal{D}(\rho) = \bigcup_{{\rho^\prime} \in \mathcal{S}} \mathcal{D}({\rho}^\prime)$; this union is disjoint.  We define $\widetilde{D}_{0,\sigma}$ to be the $KZ$-submodule of $D_{0,\sigma}({\rho}^\prime)$ generated by the ${I(1)}$-invariants, where $I(1) \leq I$ is the pro-$p$-Sylow subgroup.  Then $\widetilde{D}_{0,\sigma}$ has length two if $f = 2$.  We identify $\mathcal{S}$ with the set of vertices of an $e \times e$ square lattice and choose a Hamiltonian walk $\gamma$ through this lattice.  Definition~\ref{def:D0} modifies some of the $\widetilde{D}_{0,\sigma}$, in a way depending on $\gamma$, to obtain $KZ$-modules $D_{0,\sigma}^\gamma$ of length two with socle $\sigma$; these still arise from $\Gamma$ by inflation.  Set $D_0^\gamma(\rho) = \bigoplus_{\sigma \in \mathcal{D}(\rho)} D_{0,\sigma}^\gamma$ and $D_1^\gamma(\rho) = (D_0^\gamma(\rho))^{I(1)}$, with the natural inclusion $\iota: D_1^\gamma(\rho) \hookrightarrow D_0^\gamma(\rho)$.  In contract to the situation in~\cite{BP/12}, Serre weights may appear as subquotients of $D_0^\gamma(\rho)$ with multiplicity greater than one.  However, there is a unique way to extend the $IZ$-action on $D_1^\gamma(\rho)$ to an action of $N$, modulo the choices of isomorphisms mentioned above, if we require that no $N$-orbit be contained in $\mathrm{soc}_K D_0^\gamma (\rho)$.  This $N$-action interweaves the $KZ$-modules $D_{0,\sigma}^\gamma$ for Serre weights $\sigma$ occurring in $\mathcal{D}(\rho^\prime)$ for different Galois representations $\rho^\prime \in \mathcal{S}$ and enables us to prove, in Theorem~\ref{thm:main}, that $G$-modules $\Psi$ arising from the diagrams $(D_0^\gamma(\rho), D_1^\gamma(\rho), \iota)$ are irreducible and supersingular.  
In the course of the proof we observe that if $W \subseteq \Psi$ is a $G$-submodule and $\tau \subseteq \mathrm{soc}_K W$, then the action of $N$ on $\tau^{I(1)}$ provides elements of $D_{0,\sigma}^\gamma$, for some $\sigma \neq \tau$, that are not contained in $\mathrm{soc}_K D_{0,\sigma}^\gamma$ but are contained in $W$.  Since $D_{0,\sigma}^\gamma$ has length two (this is the reason for the hypothesis $f = 2$ in Theorem~\ref{thm:intro}) we immediately conclude that 
$D_{0,\sigma}^\gamma \subset W$, resolving the difficulty discussed at the end of the previous paragraph.

\subsection{Relation to the mod $p$ local Langlands correspondence} \label{sec:Langlands}
As mentioned above, some of the supersingular representations of $\mathrm{GL}_2(F)$ constructed in~\cite{BP/12} for unramified $F/\Q_p$ appear in local-global compatibility results for the mod $p$ Langlands correspondence.  By contrast, the diagrams constructed here for ramified $F/ \Q_p$ seem unlikely to occur in completed cohomology.  
One indication of this is the following.  Breuil~\cite{Breuil/11} proposed an extension to certain diagrams for unramified $F/\Q_p$ of the Colmez functor from $\mathrm{GL}_2(\Q_p)$-modules to $(\varphi,\Gamma)$-modules, and hence to Galois representations, realizing the mod $p$ local Langlands correspondence for $\mathrm{GL}_2(\Q_p)$.  If $D$ is a diagram associated to the irreducible representation $\rho: G_F \to \mathrm{GL}_2(\overline{\mathbb{F}}_p)$, for $F/\Q_p$ unramified, we let $M(D)$ denote the $(\varphi,\Gamma)$-module for $\Q_p$ defined in~\cite{Breuil/11}, and $V(M(D))$ the associated representation of $G_{\Q_p}$.  Then Breuil~\cite[Corollaire~5.4]{Breuil/11} proved that $V(M(D))_{| I_{\Q_p}} \simeq ( \mathrm{Ind}_{G_{F}}^{\otimes G_{\Q_p}} \rho^\vee )_{| I_{\Q_p}}$, where $I_{\Q_p} \leq G_{\Q_p}$ is the inertia subgroup.  Moreover, for certain choices of the isomorphisms defining $D$, one has that $V(M(D))$ is isomorphic to the tensor induction $\mathrm{Ind}_{G_{F}}^{\otimes G_{\Q_p}} \rho^\vee$ as representations of $G_{\Q_p}$.

As we observe in \S\ref{sec:phigamma} below, one can naturally associate a $(\varphi,\Gamma)$-module $M(D)$ for $\Q_p$ to the diagrams $D$ constructed in this paper.  In Proposition~\ref{pro:explicit.phi.gamma} we apply the computations of~\cite{Breuil/11} to give an explicit description of the associated Galois representation $V(M(D))$.  It turns out that even the restriction of $V(M(D))$ to the inertia subgroup $I_F \leq G_F \leq G_{\Q_p}$ depends on the Hamiltonian walk $\gamma$ chosen in the construction of $D$, and that in general there is no choice of $\gamma$ for which we would obtain an isomorphism $V(M(D))_{| I_F} \simeq ( \mathrm{Ind}^{\otimes G_{\Q_p}}_{G_{F_0}} ( \mathrm{Ind}^{G_{F_0}}_{G_F} \rho^\vee ) )_{| I_F}$.

The present work suggests that to obtain a generalization of the construction of~\cite{BP/12} that would be defined more canonically and would admit generalizations of the more recent local-global compatibility results, one should consider diagrams for which the $KZ$-module $D_0$ factors not through $\Gamma$ but only through a larger finite quotient.  There is work in progress in this direction.

\section{Preliminaries} \label{sec:prelim}
\subsection{Serre weights}
As in the introduction,
let $F / \Q_p$ be a finite extension, and let $F_0$ be the maximal unramified subextension.  Let $\mathcal{O}$ be the ring of integers of $F$, fix a uniformizer $\pi \in \mathcal{O}$, and let $k = \mathcal{O} / (\pi)$ denote the residue field.  Let $q = p^f$ be the cardinality of $k$.  If $\xi \in k^\times$, let $[\xi ] \in \mathcal{O}$ denote the $(q-1)$-th root of unity lifting $\xi$, and set $[0] = 0$.  A Serre weight is an irreducible $\overline{\F}_p$-representation of the finite group $\Gamma = \mathrm{GL}_2(k)$, which can be viewed as a representation of $K = \mathrm{GL}_2(\mathcal{O})$ or of $\mathrm{GL}_2(\mathcal{O}_{F_0})$ by inflation.  Every irreducible smooth $\overline{\F}_p$-representation of these two profinite groups arises from a Serre weight in this way.  Let $B \leq \Gamma$ be the subgroup of upper triangular matrices, and let $U \leq B$ be the subgroup of upper triangular matrices all of whose eigenvalues are $1$.  Let $K(1)$ be the kernel of the reduction map $K \twoheadrightarrow \Gamma$, and let $I$ and $I(1)$ be the preimages of $B$ and $U$, respectively.  

Given a Serre weight $\sigma$, write $\chi(\sigma)$ for the character by which the diagonal torus $H \leq B$ acts on $\sigma^{U}$; we also view $\chi(\sigma)$ as a character of $B$ by inflation.  For $v \in \sigma$ and $g \in G = \mathrm{GL}_2(F)$, we denote by $g \tensor v$ the element of the compact induction $\ind_{KZ}^G \sigma$ supported on the right coset $KZ g^{-1}$ and sending $g^{-1}$ to $v$.  Set the notations $\alpha = \left( \begin{array}{cc} 1 & 0 \\ 0 & \pi \end{array} \right)$ and $w = \left( \begin{array}{cc} 0 & 1 \\ 1 & 0 \end{array} \right)$, and denote $\Pi = \alpha w = \left( \begin{array}{cc} 0 & 1 \\ \pi & 0 \end{array} \right)$.  Let $Z$ be the center of $G$ and $\mathrm{Id}_2$ the identity matrix.  The normalizer $N = N_{G}(I(1))$ is generated by $IZ$ and $\Pi$.  Write $\chi^w : H \to \overline{\mathbb{F}}_p^\times$ for the character $\chi^w(h) = \chi(whw)$; this is denoted $\chi^s$ in~\cite{BP/12}.  Let $\sigma^{[w]}$ be the unique Serre weight distinct from $\sigma$ such that $\chi(\sigma^{[w]}) = \chi(\sigma)^w$.  We view Serre weights as representations of $KZ$ by letting $\pi \mathrm{Id}_2$ act trivially.

Fix an embedding $\varepsilon_0 : k \hookrightarrow \overline{\F}_p$ and define embeddings $\varepsilon_i$ for every $i \in \N$ by means of the recursion $\varepsilon_i = \varepsilon_{i-1}^p$.  Then every Serre weight has the form
\begin{equation} \label{equ:serre.weight}
\sigma = \bigotimes_{i = 0}^{f - 1} \left( \mathrm{Sym}^{r_i} k^2 \otimes_{k, \varepsilon_i} \overline{\F}_p \right) \tensor (\eta \circ \det),
\end{equation}
where $0 \leq r_i \leq p - 1$ for every $i$ and $\eta : k^\times \to \overline{\F}_p^\times$ is a character.  
We will write $\sigma = \det^m \tensor (r_0, \dots, r_{f-1})$ for $\sigma$ as in~\eqref{equ:serre.weight}, where $m \in \Z/(q-1)\Z$ is such that $\eta = \varepsilon_0^m$; then $\sigma^{[w]} = \det^{m+r} \tensor (p-1-r_0, \dots, p-1-r_{f-1})$ for $r = \sum_{i = 0}^{f-1} r_i p^i$.  An irreducible $\overline{\mathbb{F}}_p$-representation $V$ of $G$ is called {\emph{supersingular}} if $\lambda = 0$ for every surjective map (equivalently, for one such map) of the form $\ind_{KZ}^G \sigma / (T - \lambda)\ind_{KZ}^G \sigma \to V$, where $\lambda \in \overline{\mathbb{F}}_p$ and $T \in \mathrm{End}_G(\ind_{KZ}^G \sigma)$ is the operator of~\cite[Proposition~8]{BL/94}.

\begin{lemma} \label{lem:ind.structure}
Let $\sigma$ be a Serre weight such that $\chi(\sigma) \neq \chi(\sigma)^w$ and let $0 \neq v \in \sigma^{I(1)}$.  Then $\langle \alpha \tensor wv \rangle_K \subset (\ind_{KZ}^G \sigma)^{K(1)}$.  Moreover, $\langle \alpha \tensor wv \rangle_K \simeq \Ind_I^K \chi(\sigma)^w$ and we have $\mathrm{soc}_K(\langle \alpha \tensor wv \rangle_K) = T(\langle \mathrm{id} \tensor v \rangle_K)$.  
\end{lemma}
\begin{proof}
Clearly $\mathrm{Id}_2 \tensor v \in (\mathrm{ind}_{KZ}^G \sigma)^{I(1)}$, and hence $\alpha \tensor wv = \Pi (\mathrm{Id}_2 \tensor v) \in (\mathrm{ind}_{KZ}^G \sigma)^{I(1)}$, since $\Pi$ normalizes $I(1)$.  In particular, $\alpha \tensor wv$ is invariant under the action of $K(1)$, and thus $\langle \alpha \tensor wv \rangle_K \subset (\mathrm{ind}_{KZ}^G \sigma)^{K(1)}$ as $K(1) \unlhd K$. Moreover, $I$ acts on $\alpha \tensor wv$ via the character $\chi(\sigma)^w$, so by Frobenius reciprocity there is a surjection of $K$-modules $\psi: \mathrm{Ind}_I^K \chi(\sigma)^w \twoheadrightarrow \langle \alpha \tensor wv \rangle_K$.  A simple computation shows that 
$$ \left( \begin{array}{cc} [ \xi ] & 1 \\ 1 & 0 \end{array} \right) (\alpha \tensor wv) = \left( \begin{array}{cc} \pi & [ \xi ] \\ 0 & 1 \end{array} \right) \tensor v.$$
The cosets $KZ \left( \begin{array}{cc} \pi & [ \xi ] \\ 0 & 1 \end{array} \right)^{-1}$, for $\xi \in k$, and $KZ \alpha^{-1}$ are pairwise distinct; hence $\dim_{\overline{\mathbb{F}}_p} \langle \alpha \tensor wv \rangle_K \geq q + 1 = \dim_{\overline{\mathbb{F}}_p} \mathrm{Ind}^K_I \chi(\sigma)^w$ and $\psi$ is an isomorphism.  By~\cite[Theorem~2.4]{BP/12} we then have $\mathrm{soc}_K \langle \alpha \tensor wv \rangle_K \simeq \sigma$, and by~\cite[Lemma~2.7]{BP/12} a non-zero element of $(\mathrm{soc}_K \langle \alpha \tensor wv \rangle_K)^{I(1)}$ is given by
$$ \sum_{\xi \in k} \left( \begin{array}{cc} \pi & [ \xi ] \\ 0 & 1 \end{array} \right) \tensor v = T(\mathrm{Id}_2 \tensor v),$$
where the equality follows from the explicit description of $T$ in~\cite[Proposition~2.1]{Hendel/19}.
\end{proof}

Recall that the socle filtration $\{ \mathrm{soc}_i (M) \}$ of a $C$-module $M$, for any group $C$, is defined recursively by $\mathrm{soc}_0 (M) = \mathrm{soc}_C (M)$ and $\mathrm{soc}_i(M)/\mathrm{soc}_{i-1}(M) = \mathrm{soc}_C (M / \mathrm{soc}_{i-1}(M))$.  For the rest of this section, we assume $f = 2$.  Assume $p > 2$, and let $\sigma = \det^m \tensor (r_0, r_1)$ be a Serre weight with $(r_0, r_1) \not\in \{ (0,0), (p-1, p-1) \}$.  If $r_0 - a_0 + p(r_1 - a_1)$ is even, then there are two twists of $(a_0, a_1)$ with the same central character as $\sigma$.  We introduce the following notation for brevity:
\begin{eqnarray} \label{eqn:plus.notation}
(a_0, a_1)_{\sigma}^+ & =&\det\nolimits^{m+\frac{1}{2} (r_0 - a_0 + p(r_1 - a_1))} \tensor (a_0, a_1) \\
\nonumber
(a_0, a_1)_{\sigma}^- & =&\det\nolimits^{m+\frac{1}{2}(p^2 - 1 + r_0 - a_0 + p(r_1 - a_1))} \tensor (a_0, a_1).
\end{eqnarray}
In particular, $\sigma = (r_0, r_1)_\sigma^+$ and $\sigma^w = (p-1-r_0, p-1-r_1)_\sigma^-$.
\begin{lemma} \label{lem:socle.filtration}
Let $\sigma = \det^m \tensor (r_0, r_1)$ be a Serre weight.  If $(r_0, r_1) \not\in \{(0,0), (p-1,p-1) \}$,  then the socle filtration of $\mathrm{Ind}_B^\Gamma \chi(\sigma)^w$ is as follows, reading from left to right:
$$ (r_0, r_1)_\sigma^+ \, {\textbf{---}} \, (p-2-r_0, r_1 - 1)_\sigma^+ \oplus (r_0 - 1, p - 2 - r_1)_\sigma^- \, {\textbf{---}} \, (p-1-r_0, p-1-r_1)_\sigma^-.$$
\end{lemma}
\begin{proof}
This is a special case of~\cite[Theorem~2.4]{BP/12}, which itself was originally established, in somewhat different terminology, by Bardoe and Sin~\cite[Theorem~C]{BS/00}.
\end{proof}

Given a Serre weight $\sigma = \det^m \tensor (r_0,r_1)$ such that $0 \leq r_0, r_1 \leq p - 2$, define $Q_{ \{ 0 \}}(\sigma)$ and $Q_{ \{ 1 \}}(\sigma)$ to be the quotients having socle $\sigma$ of $\Ind_B^\Gamma \chi((p-2-r_0, r_1 + 1)_{\sigma}^+)$ and $\Ind_B^\Gamma \chi((r_0 + 1, p-2-r_1)_{\sigma}^-)$, respectively.  It follows from Lemma~\ref{lem:socle.filtration} that these are $\Gamma$-modules of length two with socle filtration
\begin{align*}
Q_{\{ 0 \} }(\sigma):& \,  &\det\nolimits^m \tensor (r_0, r_1)& \, &\textbf{---}& \, &\det\nolimits^{m+r_0 + 1 - p} \tensor (p-2-r_0, r_1 + 1) \\
Q_{\{ 1 \} }(\sigma):& \, &\det\nolimits^m \tensor (r_0, r_1)& \, &\textbf{---}& \, &\det\nolimits^{m + p r_1 + p - 1} \tensor (r_0 + 1, p - 2 - r_1).
\end{align*}

\begin{lem} \label{lem:two.dim.invariants}
Let $\sigma = \det^m \tensor (r_0, r_1)$ be a Serre weight so that $0 \leq r_0, r_1 \leq p - 2$.  Then 
$$\dim_{\overline{\mathbb{F}}_p} Q_{\{ 0 \} }(\sigma)^{U} = \dim_{\overline{\mathbb{F}}_p} Q_{\{ 1 \} }(\sigma)^{U} = 2.$$
\end{lem}  
\begin{proof}
The unique irreducible submodule of $Q_{\{ 0 \}}(\sigma)$ contains a one-dimensional subspace of $U$-invariants~\cite[Lemma~2]{BL/94}.  It is easy to see that $\mathrm{Ind}_B^\Gamma \chi$, for any character $\chi$, is generated by the function $\varphi$ supported on $B$ such that $\varphi(b) = \chi(b)$ for all $b \in B$.  Clearly $\varphi$ is invariant under the action of $U$.  Hence $Q_{\{ 0 \}}(\sigma)$, which is a quotient of $\mathrm{Ind}_B^\Gamma \chi$ for a suitable $\chi$, can be generated by a $U$-invariant.  Thus $\dim_{\overline{\mathbb{F}}_p} Q_{\{ 0 \}}(\sigma)^U \geq 2$.  On the other hand, $Q_{\{ 0 \}}(\sigma)$ is an extension of two irreducible modules, each of which has a one-dimensional space of $U$-invariants.  Thus $\dim_{\overline{\mathbb{F}}_p} Q_{\{ 0 \}}(\sigma)^U \leq 2$.  The same argument applies to $Q_{\{ 1 \}}(\sigma)$.
\end{proof}

\subsection{Generic Galois representations and modular Serre weights}
Let $G_F$ be the absolute Galois group of $F$ and $I_F \leq G_F$ the inertia subgroup.
Let $k^\prime$ be the quadratic extension of $k$, let $F^{\mathrm{nr}}$ be the maximal unramified extension of $F$, and let $L^\prime / F^{\mathrm{nr}}$ be totally ramified such that $\mathrm{Gal}(L^\prime / F^{\mathrm{nr}}) \simeq (k^\prime)^\times$.  Consider the natural projection $\nu^\prime : I_F \to \mathrm{Gal}(L^\prime / F^{\mathrm{nr}})$.
Let $\omega_{2f} = \varepsilon_0^\prime \circ \nu^\prime : I_F \to \overline{\mathbb{F}}_p^\times$ be a fundamental character of level $2f$ corresponding to an embedding $\varepsilon_0^\prime: k^\prime \hookrightarrow \overline{\mathbb{F}}_p$ that restricts to $\varepsilon_0$ on $k$.

\begin{dfn} \label{def:generic}
An irreducible representation $\rho: G_F \to \mathrm{GL}_2(\overline{\mathbb{F}}_p)$ 
is {\emph{generic}} if $\rho_{| I_F}$ is isomorphic to a twist of
\begin{equation} \label{equ:form}
\omega_{2f}^{\sum_{i = 0}^{f-1} p^i (r_i + 1)} \oplus \omega_{2f}^{q \sum_{i = 0}^{f-1} p^i (r_i + 1)}
\end{equation}
with $2e - 1 \leq r_0 \leq p - 2$ and $2e - 2 \leq r_i \leq p - 3$ for $i > 0$.  
\end{dfn}

Note that if $e = 1$, then Definition~\ref{def:generic} coincides with the notion of genericity of~\cite[Definition~11.7]{BP/12}.
If $\rho: G_F \to \mathrm{GL}_2(\overline{\F}_p)$ is irreducible and $\mathcal{D}(\rho)$ is the set of Serre weights associated to it in~\cite[\S2]{Schein/08}, then by~\cite[Proposition~3.1]{Schein/09} $\mathcal{D}(\rho)$ is a union of sets of Serre weights associated to irreducible representations of $G_{F_0}$ in~\cite[\S3.1]{BDJ/10}.  If $\rho$ is generic, then there are $e^f$ such representations of $G_{F_0}$, they are all generic, and the associated sets of Serre weights each have $2^f$ elements, are pairwise disjoint, and are described in~\cite[\S11]{BP/12}.  We now describe $\mathcal{D}(\rho)$ explicitly in the case that we will treat.

\begin{asm}
{\emph{For the rest of this note, assume $f = 2$ and that $\rho: G_F \to \mathrm{GL}_2(\overline{\mathbb{F}}_p)$ is irreducible and generic.}}  
\end{asm}

Suppose that $\rho_{| I_F}$ is~\eqref{equ:form} twisted by $\omega_{4}^{m(p^2 + 1)}$, and fix the Serre weight $\tau = \det^m \tensor (r_0, r_1)$.
Consider the set
$\Delta = \{ (\delta_0, \delta_{1}) : 0 \leq \delta_0, \delta_1 \leq e - 1 \}$.  For every $\underline{\delta} = (\delta_0, \delta_1) \in \Delta$, let $\mathcal{D}(\underline{\delta})$ be the set of four Serre weights appearing in the socle of~\eqref{equ:schematic} below.  Then $\mathcal{D}(\rho)$ is the disjoint union
\begin{equation} \label{equ:decomposition}
\mathcal{D}(\rho) = \coprod_{\underline{\delta} \in \Delta} \mathcal{D}(\underline{\delta}).
\end{equation}
We associate to every $\sigma \in \mathcal{D}(\rho)$ a pair $(\underline{\delta}_\sigma, J_\sigma)$, where $\underline{\delta}_\sigma \in \Delta$ is determined by $\sigma \in \mathcal{D}(\underline{\delta}_{\sigma})$, whereas $J_\sigma \subseteq \{ 0, 1 \}$ 
is the set associated to $\sigma$ in~\cite[\S11]{BP/12}.  These sets are $\varnothing$, $\{ 0 \}$, $\{ 1 \}$, and $\{ 0, 1 \}$, by order of appearance in~\eqref{equ:schematic}.

For $\underline{\delta} \in \Delta$, let $\widetilde{D}_0(\underline{\delta})$ be the $\Gamma$-module generated by the $U$-invariants of the $\Gamma$-module $D_0(\rho^\prime)$ associated to $\mathcal{D}(\underline{\delta})$ in~\cite[\S13]{BP/12}; here $\rho^\prime$ is an irreducible representation of $G_{F_0}$ such that $\mathcal{D}(\rho^\prime) = \mathcal{D}(\underline{\delta})$.  Then $\widetilde{D}_0(\underline{\delta})$ is the following direct sum of four $\Gamma$-modules: 
\begin{eqnarray} \label{equ:schematic}
{(r_0 - 2 \delta_0, r_1 - 2 \delta_1)_{\tau}^+} \, & \textbf{---} \, & (r_0 - 2 \delta_0 + 1, p - r_1 + 2 \delta_1 - 2)_{\tau}^- \\ 
\nonumber & \oplus & \\ \nonumber
{(r_0 - 2 \delta_0 - 1, p - r_1 + 2 \delta_1 - 2)_{\tau}^-} \,& \textbf{---} \, & (p - r_0 + 2 \delta_0 - 1, p - r_1 + 2 \delta_1 - 1)_{\tau}^-  \\ \nonumber & \oplus & \\ \nonumber
{(p - r_0 + 2 \delta_0 - 2, r_1 - 2 \delta_1 + 1)_{\tau}^+} \, & \textbf{---} \, & (r_0 - 2 \delta_0, r_1 - 2 \delta_1 + 2)_{\tau}^+ \\ \nonumber & \oplus & \\ \nonumber
{(p - r_0 + 2 \delta_0 - 1, p - r_1 + 2 \delta_1 - 3)_{\tau}^-} \, & \textbf{---} \,  & (p - r_0 + 2 \delta_0, r_1 - 2 \delta_1 + 1)_{\tau}^+.
\end{eqnarray}
Observe that $\widetilde{D}_0(\underline{\delta}) = \bigoplus_{\sigma \in \mathcal{D}(\underline{\delta})} \widetilde{D}_{0,\sigma}$, where $\widetilde{D}_{0,\sigma} = Q_{\{ 0 \}}(\sigma)$ if $J_\sigma = \{ 0 \}$ or $J_\sigma = \{ 1 \}$, and $\widetilde{D}_{0,\sigma} = Q_{\{ 1 \}}(\sigma)$ otherwise.

\begin{rem} \label{rem:reasons}
If $e > (p-1)/2$, then no generic Galois representations $\rho$ exist.  We restrict our attention to generic $\rho$ for several reasons.  Firstly, although it is known by~\cite[Theorem~A]{GK/14} 
that $\mu_\sigma(\rho) > 0$ if and only if $\sigma \in \mathcal{D}(\rho)$, provided one believes the Breuil-M\'{e}zard conjecture and hence that the multiplicities $\mu_{\sigma}(\rho)$ exist, in the case of $F / \Q_p$ ramified nothing has yet been proved about the value of $\mu_\sigma(\rho)$ when it is positive.
However, for generic $\rho$ it is expected that $\mu_{\sigma}(\rho) = 1$ for all $\sigma \in \mathcal{D}(\rho)$; thus for generic $\rho$ we are able to specify in~\eqref{equ:good.socle} the $K$-socle of the desired supersingular representations.

Secondly, if $\rho$ is non-generic, then some of the representations $\rho^\prime$ of $G_{F_0}$ appearing in the right-hand side of the analogue of~\eqref{equ:decomposition} are non-generic, and the associated sets $\mathcal{D}(\rho^\prime)$ of Serre weights do not admit the shape of the socle of~\eqref{equ:schematic}.  See~\cite[\S4.1.1]{David/17} for an explicit description of $\mathcal{D}(\rho^\prime)$ in non-generic cases.  Moreover, if $\rho$ is non-generic, then $\mathcal{D}(\rho)$ contains Serre weights $\sigma$ of the form~\eqref{equ:serre.weight} with $r_0 = p-1$ or $r_1 = p - 1$.  Such Serre weights do not arise in $\mathrm{soc}_1(\mathrm{Ind}^{\Gamma}_B \chi)$ for any character $\chi$, and hence $Q_{\{ 0 \}}(\sigma)$ and $Q_{\{ 1 \}}(\sigma)$ cannot be defined.  The uniform construction of the following section is thus limited to diagrams whose $K$-socle has the form $\bigoplus_{\sigma \in \mathcal{D}(\rho)} \sigma$ for generic $\rho$.
\end{rem}

\section{Families of diagrams}
\subsection{Construction} \label{sec:construction}
We follow the spirit, although not the actual technique, of the constructions of non-admissible irreducible $\overline{\mathbb{F}}_p$-representations of $G$ in~\cite{Le/19, GS/20} (see also the expository~\cite{GS/19}) to obtain an irreducible diagram from $\widetilde{D}_0(\rho) = \bigoplus_{\underline{\delta} \in \Delta} \widetilde{D}_0(\underline{\delta})$ by defining an action of $\Pi$ on $\widetilde{D}_{0}(\rho)^{I(1)}$ that interweaves the direct summands $\widetilde{D}_0(\underline{\delta})$.  First we replace $\widetilde{D}_0(\rho)$ with a new $\Gamma$-module $D_0(\rho)$ by modifying some of the components $\widetilde{D}_{0,\sigma}$.

If $\sigma \in \mathcal{D}(\rho)$, let $\kappa(\sigma)$ denote the $\Gamma$-cosocle of $\widetilde{D}_{0, \sigma}$.  Since $\widetilde{D}_{0,\sigma}$ is a non-split extension of two irreducible $\Gamma$-modules, $\kappa(\sigma)$ is irreducible.  The following is obvious by inspection.

\begin{lemma}
Let $\sigma, \tau \in \mathcal{D}(\rho)$.  Then $\kappa(\sigma) \simeq \kappa(\tau)^{[w]}$ if and only if one of the following holds:
\begin{enumerate}
\item The pairs associated to $\sigma$ and $\tau$ are $((\delta_0, \delta_1), \{ 0 \})$ and $((\delta_0, \delta_1 + 1), \{ 1 \})$ for some $(\delta_0, \delta_1) \in \Delta$.
\item The pairs associated to $\sigma$ and $\tau$ are $((\delta_0, \delta_1), \{ 0,1 \})$ and $((\delta_0 + 1, \delta_1), \varnothing)$ for some $(\delta_0, \delta_1) \in \Delta$.
\end{enumerate}
\end{lemma}

Consider the graph with vertex set $\Delta$, where two vertices $(\delta_0, \delta_1)$ and $(\delta_0^\prime, \delta_1^\prime)$ are adjacent if $(\delta_0^\prime, \delta_1^\prime) \in \{ (\delta_0 \pm 1, \delta_1), (\delta_0, \delta_1 \pm 1) \}$; this is an $e \times e$ square lattice.  Fix a Hamiltonian walk $\gamma$ in this graph, namely an undirected path that traverses each vertex exactly once, with no restriction on the starting and ending vertices.  It is clear that such paths exist.  We say that two adjacent elements of $\Delta$ are $\gamma$-adjacent if $\gamma$ contains the edge connecting them.  

\begin{dfn} \label{def:D0}
Let $\sigma \in \mathcal{D}(\rho)$ be associated to the pair $((\delta_0, \delta_1), J)$, and let $\gamma$ be a Hamiltonian walk.  Define a $\Gamma$-module $D^\gamma_{0,\sigma}$ as follows:
$$
D^\gamma_{0,\sigma} = \begin{cases}
Q_{\{ 1 \} }(\sigma) &: (\delta_0, \delta_1) \text{ is $\gamma$-adjacent to } (\delta_0, \delta_1 + 1) \text{ and } J = \{ 0 \} \\
Q_{\{ 1 \} }(\sigma) &: (\delta_0, \delta_1) \text{ is $\gamma$-adjacent to } (\delta_0, \delta_1 - 1) \text{ and } J = \{ 1 \} \\
Q_{\{ 0 \} }(\sigma) &: (\delta_0, \delta_1) \text{ is $\gamma$-adjacent to } (\delta_0 + 1, \delta_1) \text{ and } J = \{ 0, 1 \} \\
Q_{\{ 0 \} }(\sigma) &: (\delta_0, \delta_1) \text{ is $\gamma$-adjacent to } (\delta_0 - 1, \delta_1) \text{ and } J = \varnothing \\
\widetilde{D}_{0,\sigma} &: \text{ otherwise}.
\end{cases}
$$
Set $D_0^\gamma(\rho) = \bigoplus_{\sigma \in \mathcal{D}(\rho)} D^\gamma_{0,\sigma}$.  We usually write $D_0(\rho)$ for $D_0^\gamma(\rho)$ to lighten the notation.
\end{dfn}

Informally, if we view $\widetilde{D}_0(\rho)$ schematically as in~\eqref{equ:schematic}, then to obtain $D_{0}(\rho)$ we switch one Serre weight appearing in the cosocle of $\widetilde{D}_0(\underline{\delta})$ with a Serre weight appearing in the cosocle of $\widetilde{D}_0(\underline{\delta}^\prime)$ whenever $\underline{\delta}$ and $\underline{\delta}^\prime$ are $\gamma$-adjacent.  If $e = 1$, then $\Delta$ has only one vertex and there are no switches, so the $D_0(\rho)$ constructed here is the $\Gamma$-submodule of the $\Gamma$-module $D_0(\rho)$ constructed in~\cite{BP/12} generated by its $U$-invariants.  

\begin{rem}
Observe that the $\Gamma$-module $D_0(\rho)$ defined above is a direct sum of $4e^2$ non-split extensions of two Serre weights.  Moreover, since $D_0(\rho)$ is a direct sum of $\Gamma$-modules of the form~\eqref{equ:schematic}, up to permutation of Serre weights appearing in the cosocle, it is clear that a Serre weight $\sigma$ appears in the socle of $D_0(\rho)$ if and only if $\sigma^{[w]}$ appears in the cosocle.  If $e > 1$, then there exist Serre weights $\sigma$ such that $\{ \sigma, \sigma^{[w]} \} \subset \mathcal{D}(\rho)$; see Example~\ref{exm:example} below.  In this case, $\sigma$ and $\sigma^{[w]}$ each appears both in the socle and in the cosocle of $D_0(\rho)$.  Thus, unlike the situation of~\cite[\S14]{BP/12}, the $\Gamma$-module $D_0(\rho)$ is not multiplicity-free when $F / \Q_p$ is ramified.
\end{rem}

Observe that $\dim_{\overline{\mathbb{F}}_p} D_0(\rho)^U = 8e^2$ by Lemma~\ref{lem:two.dim.invariants}, and that a $4e^2$-dimensional subspace lies inside $\mathrm{soc}_{\Gamma}  D_0(\rho)$. 
We have the following analogue of~\cite[Corollary~13.6]{BP/12}.

\begin{lemma} \label{lem:partition}
There is a unique partition of the $B$-eigencharacters of $D_0(\rho)^U$ into pairs $\{ \chi, \chi^w \}$, with $\chi \neq \chi^w$, such that one character of each pair arises from the socle of $D_0(\rho)$ and the other from the cosocle.
\end{lemma}
\begin{proof}
By inspection of~\eqref{equ:schematic}, the claim holds for the $B$-eigencharacters of $\widetilde{D}_0(\underline{\delta})^U$ for every $\underline{\delta} \in \Delta$.  Since the sets $\mathcal{D}(\underline{\delta})$ are disjoint, the claim remains true for the $B$-eigencharacters of $\widetilde{D}_0(\rho)^U$.  Now the sets of $B$-eigencharacters arising from the socle and cosocle of $D_0(\rho)$ are the same as for $\widetilde{D}_0(\rho)$, completing the proof.
\end{proof}
As noted above, if $e > 1$, then $\mathcal{D}(\rho)$ contains pairs $\{ \sigma, \sigma^{[w]} \}$ of Serre weights.  Thus the uniqueness in Lemma~\ref{lem:partition} fails without the requirement that one character of each pair $\{ \chi, \chi^w \}$ come from the socle and one from the cosocle.

View $D_0^\gamma(\rho)$ as a $KZ$-module by inflation, with $\pi \mathrm{Id}_2$ acting trivially.  The partition of Lemma~\ref{lem:partition} gives rise to a family of basic $0$-diagrams $(D_0^\gamma(\rho), \{ \, \})$ in the sense of~\cite[Definition~13.7]{BP/12}, for each choice of Hamiltonian walk $\gamma$.   

\subsection{Irreducible supersingular representations}
It remains to show that the families of diagrams that we have just constructed give rise to irreducible supersingular representations of $G$.  First we show that the families are indecomposable.

\begin{pro} \label{pro:indecomposable}
The family $(D_0(\rho), \{ \, \})$ cannot be written as a direct sum of two non-zero families of diagrams.
\end{pro}
\begin{proof}
For every $\sigma \in \mathcal{D}(\rho)$, let $\beta(\sigma) \in \mathcal{D}(\rho)$ be the Serre weight such that $\chi(\sigma)^w$ arises in the cosocle of $D^\gamma_{0,\beta(\sigma)}$.  Note that $\beta(\sigma)$ is well-defined by Lemma~\ref{lem:partition} and that $\beta$ is a bijection of $\mathcal{D}(\rho)$ onto itself.  It suffices to show that the action of $\Z \simeq \langle \beta \rangle$ on the set $\mathcal{D}(\rho)$ is transitive.  Choose a direction of the path $\gamma$, which amounts to fixing a numbering $(\gamma_1, \dots, \gamma_{e^2})$ of the elements of $\Delta$ such that $\gamma_i$ and $\gamma_{i+1}$ are $\gamma$-adjacent for every $1 \leq i \leq e^2 - 1$.  For every $\widetilde{D}_0(\underline{\delta})$, the analogously defined $\langle \beta \rangle$-action is transitive; this follows by inspection of~\eqref{equ:schematic}, observing that $\beta$ permutes the subsets of $\{ 0, 1 \}$ associated to the elements of $\mathcal{D}(\underline{\delta})$ in the order $\varnothing \mapsto \{ 0 \} \mapsto \{ 0, 1 \} \mapsto \{ 1 \} \mapsto \varnothing$. Alternatively, this follows from~\cite[Theorem~15.4]{BP/12}, noting that the pairing $\{ \, \}$ there matches characters arising from the socle with characters arising from the cosocle.  
Since $\gamma_1$ is $\gamma$-adjacent to only one other element of $\Delta$, there is a unique Serre weight $\tau \in \mathcal{D}(\gamma_1)$ such that $D^\gamma_{0,\tau} \not\simeq \widetilde{D}_{0,\tau}$.  The previous observation implies that $\{ \tau, \beta(\tau), \beta^2(\tau), \beta^3(\tau) \} = \mathcal{D}(\gamma_1)$.

We proceed by induction.  Suppose it is known that all elements of $\mathcal{D}(\gamma_{k-1})$ lie in the same orbit as $\tau$.  The same holds for $\mathcal{D}(\gamma_k)$ by an easy but tedious analysis of cases.  For instance, if $\gamma_k = \gamma_{k-1} + (0,1)$ and $\gamma_{k+1} = \gamma_k + (1,0)$, then $\beta(\gamma_{k-1}, \varnothing) = (\gamma_k, \{ 1 \})$ and $\beta(\gamma_k, \{ 1 \}) = (\gamma_k, \varnothing)$ and $\beta(\gamma_k, \varnothing) = (\gamma_k, \{ 0 \})$, whereas 
$\beta^{-1}(\gamma_{k-1}, \{ 0 \}) = (\gamma_k, \{ 0,1 \})$; the other cases, including the case where $\gamma_k$ is the terminal vertex of $\gamma$, are treated similarly.
\end{proof}

\begin{rem} \label{rem:symmetry.break}
The choice of Hamiltonian walk $\gamma$ in the definition of $D_0(\rho)$ really is necessary.  It would be more canonical to start with $\widetilde{D}_0(\rho)$ and switch Serre weights in the cosocle for every pair of adjacent elements of $\Delta$.  However, the family of diagrams obtained in this way is easily checked to be decomposable in general.  
\end{rem}

Observe that if $V$ is any smooth representation of $G$, then the triple $(V^{K(1)}, V^{I(1)}, \mathrm{can})$, where $\mathrm{can} : V^{I(1)} \hookrightarrow V^{K(1)}$ is the natural inclusion, is a diagram.

The following result completes the proof of Theorem~\ref{thm:intro}.  Smooth admissible representations of $G$ satisfying its hypotheses exist by the proof of~\cite[Theorem~19.8]{BP/12}.  

\begin{thm} \label{thm:main}
Let $(D_0(\rho), D_1(\rho), r)$ be a basic $0$-diagram arising from the family constructed above.  Let $V$ be a smooth admissible representation of $G = \mathrm{GL}_2(F)$ satisfying the following conditions:
\begin{enumerate}
\item $\mathrm{soc}_K (V) = \bigoplus_{\sigma \in \mathcal{D}(\rho)} \sigma$.
\item $(D_0(\rho), D_1(\rho), r) \hookrightarrow (V^{K(1)}, V^{I(1)}, \mathrm{can})$.
\item $V$ is generated by $D_0(\rho)$.
\end{enumerate}
Then $V$ is irreducible and supersingular.  
\end{thm}
\begin{proof}
Let $W \subseteq V$ be a non-zero $G$-submodule, let $\sigma$ be a Serre weight contained in $\mathrm{soc}_K(W)$, and let $0 \neq v \in \sigma^{I(1)}$.  By Frobenius reciprocity, the inclusion $\varphi: \sigma \hookrightarrow W_{|K}$ corresponds to a non-zero map $\psi: \ind_{KZ}^G \sigma \to W$ of $G$-modules with $\psi(\mathrm{id} \tensor v) = \varphi(v)$.  Hence $\psi(\alpha \tensor wv) = \Pi(\varphi(v))$ generates $D^\gamma_{0,\beta(\sigma)}$ and in particular $\beta(\sigma) \subseteq \mathrm{soc}_K(W)$.  Proposition~\ref{pro:indecomposable} ensures that by iterating this procedure we obtain $D^\gamma_{0,\sigma} \subset W$ for all $\sigma \in \mathcal{D}(\rho)$.  Since $D_0(\rho)$ generates $V$, this implies $W = V$.  Thus $V$ is irreducible.

Moreover, the restriction to $\langle \alpha \tensor wv \rangle_K$ of the map $\psi$ above has image $\langle \Pi(\varphi(v)) \rangle_K = D^\gamma_{0,\beta(\sigma)}$, which is a $K$-module of length two by construction.  We have $\langle \alpha \tensor wv \rangle_K \simeq \mathrm{Ind}_I^K \chi(\sigma)^w$ by Lemma~\ref{lem:ind.structure}, which is a $K$-module of length $4$ by Lemma~\ref{lem:socle.filtration}.  Hence $T(\mathrm{id} \tensor v) \in \mathrm{soc}_K(\langle \alpha \tensor wv \rangle_K) \subset \ker \psi$,  and $\psi$ factors through $\ind_{KZ}^G \sigma / T(\ind_{KZ}^G \sigma)$.  Moreover, $\psi$ is surjective since $V$ is irreducible.  Hence $V$ is supersingular.
\end{proof}

We now show that different choices of Hamiltonian walks in the construction of $D_0(\rho)$ produce disjoint families of irreducible supersingular representations.

\begin{pro}
Let $\gamma$ and $\gamma^\prime$ be distinct Hamiltonian walks.
Let $V$ and $V^\prime$ be smooth admissible representations of $G = \mathrm{GL}_2(F)$ satisfying the hypotheses of Theorem~\ref{thm:main} for diagrams arising from the families $(D_0^\gamma(\rho), \{ \, \})$ and $(D_0^{\gamma^\prime}(\rho), \{ \, \})$, respectively.  Then $V$ and $V^\prime$ are not isomorphic.
\end{pro}
\begin{proof}
There exist $\underline{\delta}, \underline{\delta}^\prime \in \Delta$ which are $\gamma$-adjacent but not $\gamma^\prime$-adjacent.  Let $\beta$ and $\beta^\prime$ be the permutations of $\mathcal{D}(\rho)$ defined in the proof of Proposition~\ref{pro:indecomposable} for the families $(D_0^\gamma(\rho), \{ \, \})$ and $(D_0^{\gamma^\prime}(\rho), \{ \, \})$, respectively.
There is a unique Serre weight $\sigma \in \mathcal{D}(\underline{\delta})$ such that $\beta(\sigma) \in \mathcal{D}(\underline{\delta}^\prime)$.  Let $0 \neq v \in \sigma^{I(1)} \subset (\mathrm{soc}_K V)^{I(1)}$.  As in the proof of Theorem~\ref{thm:main}, we have $\langle \Pi v \rangle_K \simeq D^\gamma_{0, \beta(\sigma)}$.
If there is a $G$-isomorphism $f : V \stackrel{\sim}{\to} V^\prime$, then $f(v) \in \sigma^{I(1)} \subset (\mathrm{soc}_K V^\prime)^{I(1)}$ and thus $D^\gamma_{0, \beta(\sigma)} \simeq \langle \Pi f(v) \rangle_K \simeq D^{\gamma^\prime}_{0, \beta^\prime(\sigma)}$.  Taking $K$-socles, we find that $\beta(\sigma) = \beta^\prime(\sigma)$, but this is impossible since $\underline{\delta}$ is not $\gamma^\prime$-adjacent to $\underline{\delta}^\prime$, and thus $\beta^\prime(\sigma) \not\in \mathcal{D}(\underline{\delta}^\prime)$.
\end{proof}

\begin{exm} \label{exm:example}
We write out $D^\gamma_0(\rho)$ in the case $e = 2, f = 2$, where $\gamma$ is the Hamiltonian walk in $\Delta$ consisting of solid edges in the diagram below.
\begin{equation} \label{equ:ham.path}
{\small {
\begin{diagram}
(0,1) & \rLine & (1,1) \\
\dDots & & \dLine \\
(0,0) & \rLine & (1,0)
\end{diagram}
}} 
\end{equation}
The subscript $n$ of each Serre weight $\tau \in \mathcal{D}(\rho)$ is such that $\tau = \beta^n((0,0),\varnothing)$, illustrating Proposition~\ref{pro:indecomposable}.  The subscripts on the cosocles indicate the matching of Lemma~\ref{lem:partition}.  Here we use the notation of~\eqref{eqn:plus.notation} for $\sigma = \det^m \tensor (r_0,r_1)$, but we drop the subscript $\sigma$.

\begin{longtable}{LRCL}
((0,0),\varnothing) & (r_0,r_1)^+_0 & \textbf{---} \, & (r_0 + 1, p-2-r_1)^-_{15} \\
((0,0), \{ 0 \}) & (r_0 - 1, p-2-r_1)^-_{1} & \textbf{---} \, & (p-1-r_0, p-1-r_1)^-_{0} \\
((0,0), \{ 1 \}) & (p-2-r_0, r_1 + 1)^+_{15} & \textbf{---} \, & (r_0, r_1 + 2)^+_{14} \\
((0,0), \{ 0,1 \}) & (p-1-r_0, p-3-r_1)^-_{14} & \textbf{---} \, & (r_0 - 1, p-2-r_1)^-_{13} \\
((0,1), \varnothing) & (r_0, r_1 - 2)^+_6 & \textbf{---} \, & (r_0 + 1, p-r_1)^-_5 \\
((0,1), \{ 0 \}) & (r_0-1, p-r_1)^-_7 & \textbf{---} \, & (p-1-r_0, p+1-r_1)^-_6 \\
((0,1), \{ 1 \}) & (p-2-r_0, r_1 - 1)^+_5 & \textbf{---} \, & (r_0, r_1)^+_4 \\
((0,1), \{ 0,1 \}) & (p-1-r_0, p-1-r_1)^-_4 & \textbf{---} \, & (r_0 - 1, p - r_1)^-_3 \\
((1,0), \varnothing) & (r_0 - 2, r_1)^+_2 & \textbf{---} \, & (p-r_0, r_1 + 1)^+_1 \\
((1,0), \{ 0 \}) & (r_0 - 3, p-2-r_1)^-_{11} & \textbf{---} \, & (r_0 - 2, r_1)^+_{10} \\
((1,0), \{ 1 \}) & (p-r_0, r_1 + 1)^+_{13} & \textbf{---} \, & (r_0 - 2, r_1 + 2)^+_{12} \\
((1,0), \{ 0,1 \}) & (p+1-r_0, p-3-r_1)^-_{12} & \textbf{---} \, & (p+2-r_0, r_1+1)^+_{11} \\
((1,1), \varnothing) & (r_0 - 2, r_1 - 2)^+_8 & \textbf{---} \, & (p-r_0, r_1 - 1)^+_7 \\
((1,1), \{ 0 \}) & (r_0 - 3, p-r_1)^-_9 & \textbf{---} \, & (p+1-r_0, p+1-r_1)^-_8 \\
((1,1), \{ 1 \}) & (p-r_0, r_1 - 1)^+_3 & \textbf{---} \, & (p+1-r_0, p-1-r_1)^-_2 \\
((1,1), \{ 0, 1 \}) & (p+1-r_0, p-1-r_1)^-_{10} & \textbf{---} \, & (p+2-r_0, r_1 - 1)^+_9
\end{longtable}
\end{exm}

\subsection{Associated $(\varphi, \Gamma)$-modules} \label{sec:phigamma}
Breuil~\cite{Breuil/11} associated \'{e}tale $(\varphi, \Gamma)$-modules for $\Q_p$ to the diagrams of~\cite{BP/12} by adapting the Colmez functor realizing the mod $p$ local Langlands correspondence for $\mathrm{GL}_2(\Q_p)$.  In this section we observe that his construction applies also to the diagrams arising from the families $(D_0^\gamma(\rho), \{ \, \})$ constructed above, for a generic irreducible Galois representation $\rho : G_F \to \mathrm{GL}_2(\overline{\F}_p)$.

Since we assume throughout that $\rho$ is generic, every $\sigma \in \mathcal{D}(\rho)$ is determined by the character $\chi(\sigma)$.  Hence if $0 \neq v \in (\mathrm{soc}_K D_0^\gamma(\rho))^{I(1)}$ is an eigenvector for the action of $I$, then the $K$-submodule generated by $v$ is irreducible.  If $\sigma \in \mathcal{D}(\rho)$ is a twist of $(a_0,a_1)$, set
\begin{align*}
s(\sigma) & =  \begin{cases}
a_0 + 1 &: D_{0,\beta(\sigma)}^\gamma = Q_{\{ 1 \} }(\beta(\sigma)) \\
p(a_1 + 1) &: D_{0,\beta(\sigma)}^\gamma = Q_{\{ 0 \} }(\beta(\sigma))
\end{cases} \\
|s(\sigma)| & =  \begin{cases}
a_0 + 1 &: D_{0,\beta(\sigma)}^\gamma = Q_{\{ 1 \} }(\beta(\sigma)) \\
a_1 + 1 &: D_{0,\beta(\sigma)}^\gamma = Q_{\{ 0 \} }(\beta(\sigma)),
\end{cases}
\end{align*}
where $\beta: \mathcal{D}(\rho) \to \mathcal{D}(\rho)$ is the bijection from the proof of Proposition~\ref{pro:indecomposable}.

\begin{lem} \label{lem:strongly.principal}
Let $D = (D_0, D_1, \iota)$ be a diagram arising from the family $(D_0^\gamma(\rho), \{ \, \})$, and suppose that $0 \neq v \in (\mathrm{soc}_K D_0^\gamma(\rho))^{I(1)}$ is an $I$-eigenvector.  Let $\sigma_v \in \mathcal{D}(\rho)$ be the Serre weight such that $v \in \sigma \subseteq D_0^\gamma(\rho)$.  Then $s = s(\sigma_v)$ is the unique
integer $0 \leq s \leq p^2 - 1$ such that
$$ S_{s}(v) = \sum_{\lambda \in k} \lambda^s \left( \begin{array}{cc} \pi & [\lambda] \\ 0 & 1 \end{array} \right) v = \sum_{\lambda \in k} \lambda^s \left( \begin{array}{cc} [\lambda] & 1 \\ 1 & 0 \end{array} \right) \Pi v$$
is a non-zero element of $(\mathrm{soc}_K D_0^\gamma(\rho))^{I(1)}$.  
\end{lem}
\begin{proof}
The claim follows from the structure of $D_0^\gamma(\rho)$ and~\cite[Lemma~2.7]{BP/12}.
\end{proof}

It is immediate from Lemma~\ref{lem:strongly.principal} that the diagram $D$ is strongly principal in the sense of~\cite[D\'{e}finition~4.3]{Breuil/11}.  Thus the construction of~\cite[Lemme~4.5]{Breuil/11}, whose details we do not recall here, provides a $(\varphi, \Gamma)$-module $M(D)$ over $\Q_p$.  
More precisely, consider the following diagram
$D^\prime = (D_0^\prime, D_1^\prime, \iota^\prime)$ for $F_0 = \Q_{p^2}$: here $D_0^\prime$ is the $\Gamma$-module $D_0$, viewed as a module over $K_0 = \GL_2(\mathcal{O}_F)$ by inflation.  Let $I_0(1) \leq K_0$ be the pro-$p$-Iwahori subgroup, and let $D_1^\prime = D_0^U = (D^\prime_0)^{I_0(1)}$, with $\left( \begin{array}{cc} 0 & 1 \\ p & 0 \end{array} \right)$ acting in the same way that $\Pi$ acts on $D_1$.
Now define $M(D)$ to be the $(\varphi, \Gamma)$ module associated to $D^\prime$ in~\cite{Breuil/11}.

The representation $V(M(D))$ of $G_{\Q_p}$ corresponding to $M(D)$ has dimension $| \mathcal{D}(\rho)| = 4e^2$ and is described by~\cite[Proposition~4.7]{Breuil/11}.  For every $d \in \N$, let $\Q_{p^{d}} / \Q_p$ be the unramified extension of degree $d$, and let $\nu_d : G_{\Q_{p^d}} \to \overline{\F}_p^\times$ be a fundamental character of level $d$ given by $\nu_d(g) = \frac{g(\sqrt[p^d - 1]{-p})}{\sqrt[p^d - 1]{-p}} \in \F_{p^d}^\times \hookrightarrow \overline{\F}_p^\times$.  The claims below are independent of the choice of embedding $\F_{p^d} \hookrightarrow \overline{\F}_p$.

\begin{pro} \label{pro:explicit.phi.gamma}
Let $D$ be a diagram arising from the family $(D_0^\gamma(\rho), \{ \, \})$.  Let $\sigma = \det^c \tensor (a_0, a_1)$ be any element of $\mathcal{D}(\rho)$.  Then
$$ V(M(D)) \simeq \mathrm{Ind}^{G_{\Q_p}}_{G_{\Q_{p^{4e^2}}}} ( \nu_{4e^2}^A \tensor_{\overline{\F}_p} \kappa) \tensor \nu_1^{-(c+a_0 + pa_1 + 2)},$$
where $A = \frac{1}{p-1} \sum_{i = 0}^{4e^2 - 1} p^{4e^2 - 1 - i} |s(\beta^i(\sigma))|$ and $\kappa$ is an unramified character of $G_{\Q_{p^{4e^2}}}$. 
\end{pro}
\begin{proof}
Immediate from~\cite[Proposition~4.7]{Breuil/11}.
Note that $\kappa$ may be made explicit.
\end{proof}

One might hope to have $V(M(D))_{| I_{F_0}} \simeq \left( \mathrm{Ind}^{\otimes G_{\Q_p}}_{G_{F_0}} ( \mathrm{Ind}^{G_{F_0}}_{G_F} \rho^\vee ) \right)_{| I_{F_0}}$; indeed, in the case where $F = F_0$ is an unramified extension of $\Q_p$, this was proved by Breuil~\cite[Corollaire~5.4]{Breuil/11}.  However, 
this does not hold in general.  
Applying Proposition~\ref{pro:explicit.phi.gamma} to the diagrams arising from Example~\ref{exm:example} and taking $\sigma = \det^m \tensor (r_0, r_1) \in \mathcal{D}(\rho)$ as in that example, where $r_0, r_1$ are determined by $\rho$ as in~\eqref{equ:form}, we obtain
\begin{multline*}
A = p^{15}(r_0 + p - 2) + p^{14}(r_0 + p - 3) + p^{13}(r_1 + p - 2) + p^{12}(2p - 2 - r_0) + p^{11}(2p - 3 - r_0) + \\  p^{10}(r_1 + p - 3) + p^9 (r_0 + p - 2) + p^8(r_0 + p - 3) + p^7(r_0 + p - 4) + \\ p^6(2p - 2 - r_1) +  p^5 (2p - 3 - r_1) + p^4 (2p - 4 - r_1) + p^3 (2p - 1 - r_0) + \\ p^2 (2p - 2 - r_0) + p (2p - 3 - r_0) + (r_1 + p - 1) = (p+1)(p^2 + 1)B,
\end{multline*}
where $B$ is an irreducible polynomial in the variables $p, r_0, r_1$ of degree $13$ in $p$.  
It is easy to check that $\mathrm{Ind}_{G_F}^{G_{F_0}} \rho^\vee$ is a direct sum of two two-dimensional representations.  Thus $(\mathrm{Ind}^{\otimes G_{\Q_p}}_{G_{F_0}} (\mathrm{Ind}_{G_F}^{G_{F_0}} \rho^\vee))_{I_{F_0}}$ is a sum of characters of level at most $4$ and cannot be isomorphic to $V(M(D))_{I_{F_0}}$.
One checks similarly that the same holds for the other three choices of Hamiltonian paths in~\eqref{equ:ham.path}.

\begin{acknowledgements} 
The author is grateful to Christophe Breuil, Toby Gee, Eknath Ghate, Jesper Lykke Jacobsen, Nathan Keller, Daniel Le, Stefano Morra, and Mihir Sheth for helpful correspondence and conversations and for comments on an earlier version of this work, and to the anonymous referee for a careful reading of the text and for suggestions that have improved the exposition.
\end{acknowledgements}

\bibliographystyle{amsplain}

\providecommand{\bysame}{\leavevmode\hbox to3em{\hrulefill}\thinspace}
\providecommand{\MR}{\relax\ifhmode\unskip\space\fi MR }

\providecommand{\MRhref}[2]{%
  \href{http://www.ams.org/mathscinet-getitem?mr=#1}{#2}
}
\providecommand{\href}[2]{#2}

\end{document}